\documentclass[a4paper, 12pt]{article}
\usepackage{amsmath}
\usepackage{amsthm}
\usepackage{amssymb}
\usepackage[T1]{fontenc}
\usepackage[sc]{mathpazo}
\usepackage{enumerate}
\usepackage{etoolbox}
\usepackage{comment}

\patchcmd{\section}{\scshape}{\bfseries}{}{}

\newcommand{\R}{{\mathbb{R}}}

\newcommand{\N}{{\mathbb{N}}}

\DeclareMathOperator{\Log}{Log}

\newtheorem{theorem}{Theorem}%[section]
\newtheorem{proposition}[theorem]{Proposition}
\newtheorem{lemma}[theorem]{Lemma}
\newtheorem{corollary}[theorem]{Corollary}

\theoremstyle{definition}

\theoremstyle{remark}
\newtheorem{remark}[theorem]{Remark}

\newcommand{\intersect}{\cap}
\renewcommand{\supset}{\supseteq}
\newcommand{\x}{\mathsf{x}}

\renewcommand{\epsilon}{\varepsilon}

\renewcommand{\subset}{\subseteq}
\newcommand{\Union}{\bigcup}
\newcommand{\union}{\cup}

\newcommand{\congruent}{\equiv}

\newcommand{\MultiIndGrade}[1]{\N^n_{#1}}
\newcommand{\PosOrthant}{\R_+^n}

\begin{document}

% Keywords command
\providecommand{\keywords}[1]
{
  \small	
  \textit{Keywords:   } #1
}

% MSC code
\providecommand{\msccode}[1]
{
  \small	
  \textit{2020 MSC:   } #1
}

\title{\vspace{-2.5 cm}\Large Positivity of polynomials on the nonnegative part of certain affine hypersurfaces}
\author{\large Colin Tan \and \large Wing-Keung To}

%\title{Positivity of polynomials on the nonnegative part of certain affine hypersurfaces}
%\author{Colin Tan \and Wing-Keung To}

\date{}
%\date{18 Feb 2026 | Version 6 (after discussion)}

\maketitle
%\numberwithin{equation}{section}

\begin{abstract}
We consider polynomials on the intersection of the closed positive orthant with the height-$1$ level hypersurface of certain polynomials with positive coefficients.  We show that any polynomial strictly positive on such a semi-algebraic set can be represented by some polynomial with only positive coefficients.  This result generalizes a result of P\'olya which corresponds to the case when the semi-algebraic set is the standard simplex.  Our proof uses the Archimedean Representation Theorem from real algebra.
\end{abstract}

%-------------------------------------------------------------------%
\let\thefootnote\relax\footnotetext{\noindent 
 \\
\noindent 
\msccode{Primary 12D99, 14P10; secondary 13J30, 26C99}
}

\section{Introduction and Statement of Result}\label{sec: introducton}
Positivstellens{\"a}tze is a central and widely studied topic in real algebraic geometry.
Recall that a \emph{Positivstellensatz} asserts the sufficiency of a polynomial's positivity on a space (typically a semialgebraic set) for its representability in terms of an algebraic expression that immediately witnesses its positivity on the space. This algebraic expression is known as a \emph{positivity certificate}. 
In P\'olya's classical Positivstellensatz, the positivity certificate is given in terms of polynomials with positive coefficients \cite{Polya1928}.
Let $n$ be a positive integer. The \emph{standard $(n -1)$-simplex} $\Delta_{n}$ in $\R^n$ is the simplex whose $n$ vertices are $(1, 0, \dots, 0)$, $(0, 1, 0, \dots, 0)$, \dots, $(0, \dots, 0, 1) \in \R^n$. 
We say that a polynomial $p \in 
\R[x_1, \dots, x_n]$ is  \emph{strictly positive on a subset $X \subset \R^n$} (and write $p > 0$ on $X$) if $p(x) > 0$ for all $x \in X$. 

\begin{theorem}[P\'olya \cite{Polya1928}] \label{thm: Polya}
Let $p\in\R[x_1, \dots, x_n]$ be a homogeneous polynomial such that $p > 0$ on $\Delta_{n}$.
Then there exists some nonnegative integer $N_o$ such that for all integers $N \geq N_o$, 
all monomial terms of degree $N + \deg(p)$ in 
$(x_1 + \cdots + x_n)^N p$ have strictly positive coefficients. 
\end{theorem}

P\'olya's original proof in \cite{Polya1928} used an analytic argument and was later reproduced in a book on inequalities \cite[pp.\ 57--60]{HLP1952}. Powers-Reznick effectivised P\'olya's proof and gave an explicit upper bound of $N_o$ (depending on $p$, of course) \cite{PowersReznick2001}. It was Berr-W\"ormann who observed that Theorem \ref{thm: Polya} is actually a concrete instance of the Archimedean Representation Theorem, a fundamental result in real algebra, originally also known as the Kadison-Dubois Theorem \cite{BerrWoermann2001}.

\paragraph{Main Result.}  Let $\PosOrthant \subset \R^n$ denote the closed positive orthant, which consists of all the points $x =(x_1, \dots, x_n)\in \R^n$ with each $x_1, \dots, x_n \ge 0$.
Given a subset $X \subset \R^n$,
its \emph{nonnegative part} $X_+$ is defined as the intersection
\begin{equation}
 X \intersect \PosOrthant=:X_+ \subset \R^n.
\end{equation}
Thus the standard simplex $\Delta_n \subset \R^n$ in Theorem \ref{thm: Polya} may also be regarded as the nonnegative part of the affine hyperplane $H$ defined by $x_1 + \cdots + x_n = 1$. 
Our objective in this paper to obtain a generalisation of Theorem \ref{thm: Polya} from the standard simplex $\Delta_n =H_+$ to a Positivstellensatz on the nonnegative part of certain affine hypersurfaces.  

Let $\R_+$ denote the semifield of nonnegative real numbers, and let $\R[\x]:=\R[x_1,\ldots,x_n]$ denote the $\R$-algebra of real polynomials in $x_1,\ldots,x_n$.
Let $\R_+[\x] \subset \R[\x]$ denote the $\R_+$-subsemialgebra of real polynomials whose coefficients are all nonnegative. Let $\N$ denote the set of nonnegative integers.  An $n$-tuple $\alpha =(\alpha_1, \dots, \alpha_n) \in \N^n$ serves as a \emph{multi-index} for the monomial $x^\alpha := x_1^{\alpha_1} \dots x_n^{\alpha_n}$ in $\R[\x]$. Under addition, $\N^n$ forms a commutative monoid.
For any given $p = \sum_{\alpha \in \N^n} a_\alpha x^\alpha\in \R_+[x]$, we let $\Log(p) := \{\alpha \in \N^n : a_\alpha > 0\}$. 
%One easily sees that the identities $\Log(p + p') = \Log(p) \union \Log(p')$ and $\Log(pp') = \Log(p) + \Log(p')$ hold for all $p, p' \in \R_+[\x]$. Here $\Log(p) + \Log(p')$ denotes the Minkowski sum of $\Log(p)$ and $\Log(p')$ in $\N^n$.  A particular instance of the latter identity is $\Log(cp) = \Log(p)$ for all $p \in \R_+[\x]$ and real $c > 0$.
For a subset $J \subset \N^n$, write $d J := \underbrace{J + \cdots + J}_{d \, \textrm{times}} \subset \N^n$ for the $d$-fold Minkowski sum of $J$ with itself, where $d \in \N$. In particular, $0 J = \{(0, \dots, 0)\}$. 
%Then one easily sees that $\Log(p^d) = d \Log(p)$ for all $p \in \R_+[\x]$ and $d \in \N$.
Let 
$
\N^n_1 := \{(1, 0, \dots, 0)$, $(0, 1, 0, \dots, 0)$, $\dots$, $(0, \dots, 0, 1)\} \subset \N^n
$.
For example, $\Log(x_1 + \cdots + x_n) = \MultiIndGrade{1}$. 
Let $\{r = 1\} \subset \R^n$ denote the level hypersurface of a polynomial $r \in \R[\x]$ of height $1$, which consists of all points $x \in \R^n$ with $r(x) = 1$.  For $f,g,h\in \R[\x] $, write $f \congruent g \pmod{h}$ if $f-g$ is divisible by $h$ in $\R[\x]$.

Our main result in this paper is the following:

\begin{theorem} \label{thm: mainTheorem}
Let $r \in \R_+[\x]$ with $\Log(r) \supseteq \MultiIndGrade{1}$ and let $f \in \R[\x]$, where $\x = (x_1, \dots, x_n)$. Then the following statements are equivalent:

\begin{enumerate}[(a)]
\item $f > 0$ on $\{r = 1\}_+$.
\item There exists some $N_o\in\N$ such that for all integers $N\geq N_o$, there exists
 some $q_N \in \R_+[\x]$ with $\Log(q_N) = \Union_{d = 0}^N d\Log(r)$ such that 
 \begin{equation} \label{eq: certificate}
f \congruent q_N \pmod{r - 1}
\end{equation}
\item There exist some $N \in \N$ and some $q_N \in \R_+[\x]$ with $\Log(q_N) = \Union_{d = 0}^N d\Log(r)$ such that \eqref{eq: certificate} holds.
\end{enumerate}
\end{theorem} 

\begin{remark}
We remark that Theorem \ref{thm: mainTheorem} actually holds for any level hypersurface of $r$ with height $c>0$ (i.e., with $\{r = 1\}_+$ in (i) and $r-1$ in \eqref{eq: certificate} replaced by $\{r = c\}_+$ and $r-c$ respectively). We leave the easy deduction from Theorem \ref{thm: mainTheorem} to this general case to the reader.
\end{remark}

By taking $r= x_1 + \cdots + x_n$ in Theorem \ref{thm: mainTheorem}, in which case $\Log(r) = \MultiIndGrade{1}$ and $\{r = 1\}_+ = \Delta_n$, one recovers P\'olya's Positivstellensatz (Theorem \ref{thm: Polya}). In the spirit of Berr-W\"ormann \cite{BerrWoermann2001}, we will give a proof of Theorem \ref{thm: mainTheorem} in Section \ref{sec: proof} using the Archimedean Representation Theorem.

The conclusion of Theorem \ref{thm: mainTheorem} need not hold if the condition \lq$\Log(r) \supseteq \MultiIndGrade{1}$\rq~is not satisfied.  As an example, take $n=1$ and
$r=x^2 \in \R_+[\x]$, noting that $\Log(r)=\{2\}\not\supset \N_1=\{1\}$. Then $f=x$ satisfies (a), but neither (b) nor (c), in Theorem \ref{thm: mainTheorem}. Indeed $\{r = 1\}_+=\{1\}$ and $f(1) = 1 > 0$, so $f$ satisfies (a). Take any $q_N \in \R_+[x]$ with $\Log(q_N) = \Union_{d = 0}^N d\Log(r) = \{0, 2, 4, \dots, 2N\}$. When 
 $f - q_N = x - q_N$ is divided by $r - 1 = x^2 - 1$, the remainder is $x - q_N(1) \neq 0$, and thus 
 \eqref{eq: certificate} does not hold. Hence $f$ does not satisfy (b) nor (c). 
%degreeand thus $f=x$ satisfies (a) in Theorem \ref{thm: mainTheorem}.  However, if $f\in\R[x]$ satisfies (b) or (c) in Theorem \ref{thm: mainTheorem} (with $r=x^2$), then the degree of every nonzero monomial term in $f$ must be even (since both $r-1$ and $q_N$ in \eqref{eq: certificate} will have such property).

Theorem \ref{thm: mainTheorem}, being denominator free, differs from previous generalisations of P\'olya's Positivstellensatz obtained by various authors. P\'olya's Positivstellensatz is what is known as a uniform-denominator Positivstellensatz, due to the presence of the multiplier $(x_1 + \cdots + x_n)^N$ in the positivity certificate \cite{Reznick1995}. Putinar-Vasilescu \cite{PutinarVasilescu1999}, Dickinson-Povh \cite{DickinsonPovh2015} and Baldi-Sinn-Telek-Weigert \cite{BSTW} have obtained various uniform-denominator Positivstellens\"atze that generalise P\'olya's Positivstellensatz. In comparison, our Positivstellensatz is denominator free.

The juxtaposition of the positivity certificate and the type of space in Theorem \ref{thm: mainTheorem} is also relatively unusual. The positivity certificate in \eqref{eq: certificate} is congruence to a polynomial with positive coefficients, akin to Positivstellens\"atze of P\'olya \cite{Polya1928} and Handelman \cite{Handelman1988}. 
%In contrast, most other Positivstellens\"atze that are concrete instances of the Archimedean Representation Theorem have positivity certificates composed from sums of squares (such as Schm\"udgen's Positivstellensatz \cite{Schmuedgen1991}) or, more recently, sums of hermitian squares (such as Putinar-Scheiderer's Positivstellensatz \cite{PutinarScheiderer2010}). 
Yet, unlike P\'olya's Positivstellensatz and Handelman's Positivstellensatz, where the space on which the given polynomial is positive is a polytope, the space $\{r = 1\}_+$ in Theorem \ref{thm: mainTheorem} is not necessarily a polytope and, in fact, not even necessarily convex. For example, when $n = 2$ and $r = x + y + x^2$ (so $\Log(r) = \{(1,0),(0,1), (2,0) \} \supset \N^2_1$), the level curve $\{r = 1\} \subset \R^2$ is a parabola whose nonnegative part is a compact 
smooth curve with endpoints $((\sqrt{5} - 1)/2, 0)$ and $(0, 1)$ that is not a line segment. 

\begin{remark}
Theorem \ref{thm: mainTheorem} can also be deduced from a Positivstellensatz of Putinar-Scheiderer for hermitian polynomials on real hypersurfaces of complex affine space \cite[Proposition 2.3]{PutinarScheiderer2010}. This deduction is similar to how Theorem \ref{thm: Polya} may be deduced from a Positivstellensatz of Quillen \cite{Quillen1968} for hermitian bihomogeneous polynomials. 
That Theorem \ref{thm: Polya} may be deduced from Quillen's Positivstellensatz has been known since Catlin-D'Angelo \cite{CatlinDangelo1996}.
In comparison, our proof in this paper has the appeal that it stays within the realm of real polynomials.
\end{remark}

\section{The Archimedean Representation Theorem}
To facilitate our proof of Theorem \ref{thm: mainTheorem} in the next section, we recall the Archimedean Representation Theorem in the following version, working over the reals $\R$ as base field (see e.g. \cite[Theorem 5.4.4]{PrestelDelzell2001}, \cite[Theorem 1.5.9]{Scheiderer2009} or \cite[Theorem 5.3.1]{Scheiderer2024}). Let $A$ be an $\R$-algebra (always unital, associative and commutative).  Recall the semifield $\R_+\subseteq\R$ of nonnegative real numbers. By an $\R_+$-subsemialgebra of $A$, we mean a subset $S \subset A$ containing $\R_+$ and closed under addition and multiplication. Note that $S$ is a convex cone contained in the underlying $\R$-vector space of $A$. Using a notion from convex geometry, an \emph{algebraic interior point} of $S \subset A$ is a point $u \in S$ such that, for every affine line $\ell $ with $\ell\subset A$ through $u$, the interval $\ell \intersect S $ contains $u$ in its interior \cite[III.1.6]{Barvinok2002}. Note that $\ell \intersect S$ is a subinterval of $\ell$ because both $\ell$ and $S$ are convex. Algebraic interior points are also known as \emph{order units} and may be defined alternatively in terms of a quasiordering $\le_S$ on $A$ associated to $S$ as follows: For $f, f' \in A$, define $f \le_S f'$ $\iff$ $f' - f \in S$. 
Then $u \in S$ is an algebraic interior point of $S$ if and only if for each $f \in A$, there exists some $C >0$ such that $f\le_S Cu$  (see e.g. \cite{GoodearlHandelman1980}, \cite{HHL1980}).
Let $S^\circ$ denote the set of algebraic interior points of $S \subset A$. We say that $S \subset A$ is \emph{archimedean} if $1 \in S^\circ$. Finally, let $\mathsf{Alg}_{\R}(A, \mathbb{R})$ denote the set of $\R$-algebra homomorphisms (always unital) from $A$ to $\R$, and let 
$\mathcal{X}(S) \subset \mathsf{Alg}_{\R}(A, \mathbb{R})$ be the subset consisting of $\R$-algebra homomorphisms $\varphi : A \to \R$ such that $\varphi(S) \subset \R_+$. 

\begin{theorem}[Archimedean Representation Theorem {\cite[Theorem 5.4.4]{PrestelDelzell2001}, \cite[Theorem 1.5.9]{Scheiderer2009} or \cite[Theorem 5.3.1]{Scheiderer2024}}] \label{thm: ArchRepThm}
Let $A$ be an $\R$-algebra and let $S\subset A$ be an archimedean $\R_+$-subsemialgebra. 
Let $f \in A$. Then $\varphi(f) > 0$ for all $\varphi \in \mathcal{X}(S)$ if any only if $f \in S^\circ$. 
\end{theorem}

\section{Proof of Theorem \ref{thm: mainTheorem}} \label{sec: proof}

%Fix $r \in \R_+[\x]$ with $\Log(r) \supseteq \MultiIndGrade{1}$ as in Theorem \ref{thm: mainTheorem}.  Consider the principal ideal $(r - 1) \subset \R[\x]$ generated by $r - 1$. From now onwards, let \begin{equation}\label{eq: Ssubsemialg} S := \R_+[\x] + (r - 1) \subset \R[\x],\end{equation} which is an $\R_+$-subsemialgebra of $\R[\x]$.
As a preparation for the proof of Theorem \ref{thm: mainTheorem}, we first note that
the identities 
$\Log(p + p') = \Log(p) \union \Log(p')$, $\Log(pp') = \Log(p) + \Log(p')$, $\Log(cp) = \Log(p)$ and $\Log(p^d) = d \Log(p)$
hold for all $p, p' \in \R_+[\x]$, $c > 0$ and $d \in \N$.  Here $\Log(p) + \Log(p')$ denotes the Minkowski sum of $\Log(p)$ and $\Log(p')$ in $\N^n$, and $d \Log(p)$ is as defined in Section \ref{sec: introducton}.

\begin{lemma} \label{lem: prep}
Let $r\in\R_+[\x]$.  For each $N \in \N$,
    there exists $g_N \in \R_+[\x]$ with $\Log(g_N) = \Union_{d = 0}^N d\Log(r)$ such that 
$
g_N \congruent 1 \pmod{r - 1}
$.
\end{lemma}

\begin{proof}
Let $N \in \N$ be given. Then $1 \congruent (1/(N + 1))\sum_{d = 0}^N r^d =: g_N \pmod{r - 1}$.
Indeed $\Log(g_N) = \Log(\sum_{d = 0}^N r^d) = \Union_{d = 0}^N \Log(r^d) = \Union_{d = 0}^N d \Log(r)$.
\end{proof}

From now onwards, we fix $r \in \R_+[\x]$ with $\Log(r) \supseteq \MultiIndGrade{1}$ as in Theorem \ref{thm: mainTheorem}.  Consider the principal ideal $(r - 1) \subset \R[\x]$ generated by $r - 1$, and let
\begin{equation}\label{eq: Ssubsemialg} S := \R_+[\x] + (r - 1) \subset \R[\x],\end{equation} which is an $\R_+$-subsemialgebra of $\R[\x]$.
Decompose $\N^n = \coprod_{d = 0}^\infty \MultiIndGrade{d}$ as a disjoint union of the subsets $\MultiIndGrade{d} := \{\alpha=(\alpha_1,\ldots,\alpha_n) \in \N^n :\, \alpha_1 + \cdots + \alpha_n = d\}$ indexed by $d \in \N$. 
For example, $\Log((x_1 + \cdots + x_n)^d) = \MultiIndGrade{d} = d \MultiIndGrade{1}$ for all $d \in \N$.

\begin{proposition} \label{lem: algebraicInteriorPoints}
Let $f \in \R[\x]$. Then the following statements are equivalent:
\begin{enumerate}[(i)]
\item $f \in S^\circ$.
\item There exists some $N_o\in\N$ such that for every integer $N\geq N_o$, 
there exists some $q_N \in \R_+[\x]$ with $\Log(q_N) = \Union_{d = 0}^N d\Log(r)$ such that \eqref{eq: certificate} holds.
\item There exist some $N \in \N$ and some $q_N \in \R_+[\x]$ with $\Log(q_N) = \Union_{d = 0}^N d\Log(r)$ such that \eqref{eq: certificate} holds.
\end{enumerate}
\end{proposition}

\begin{proof}  
(i) $\implies$ (ii): Suppose that $f \in S^\circ$. Then $1 \le_S Cf$ for some real $C > 0$. Thus $1 + p + h(r - 1) = Cf$ for some $p \in \R_+[\x]$ and some $h\in \R[\x]$. Thus $f \congruent (1/C)(1 + p) \pmod{r - 1}$.
Let $N_o := \deg(1+p)$. 
%Let $N_o := \deg(p)$ if $p\neq 0$ (resp. $N_o := 0$ if $p=0$). 
Then for each integer $N \ge N_o$, it follows from Lemma \ref{lem: prep} that there exists 
some $g_N \in \R_+[\x]$ such that $g_N \congruent 1 \pmod{r - 1}$ and $\Log(g_N) = \Union_{d = 0}^N d\Log(r)$.  Then
\begin{equation}
f \congruent \frac{1}{C}\left(g_N + p\right)
=: q_N \pmod{r - 1},
\end{equation}
so that \eqref{eq: certificate} holds with $q_N\in \R_+[\x]$. 
We proceed to prove that $\Log(q_N)= \Union_{d = 0}^N d\Log(r)$.  
Since $N\geq N_o= \deg(1+p)$, it follows that
%Since $N\geq N_o= \deg(p)$ if $p\neq 0$ (resp. $N\geq N_o= 0 $ and $\Log(p)=\emptyset$ if $p= 0$), it follows that
\begin{equation}\label{eq: logp}
\Log(p)\subset \Log(1+p) \subset \Union_{d = 0}^{N_o} \N^n_d = \Union_{d =0}^{N_o} d\N^n_1 \subset \Union_{d = 0}^N d\Log(r).
\end{equation}
%\begin{equation}\label{eq: logp}
%\Log(p) \subset \Union_{d = 0}^{N_o} \N^n_d = \Union_{d =0}^{N_o} d\N^n_1\subset \Union_{d = 0}^{N} d\N^n_1.
%\end{equation}
Hence we have 
$\Log(q_N) = \Log(g_N + p) = \Log(g_N) \union \Log(p)=\Union_{d = 0}^N d\Log(r) \union \Log(p)=  \Union_{d = 0}^N d\Log(r)$,
where the last equality follows from \eqref{eq: logp}.

\medskip
\noindent
(ii) $\implies$ (iii): Obvious.

\medskip
\noindent
(iii) $\implies$ (i): Suppose that \eqref{eq: certificate} holds for some $N \in \N$ and some $q_N \in \R_+[\x]$ with $\Log(q_N) = \Union_{d = 0}^N d\Log(r)$. 
Let $a $ be the constant term of $q_N$.  Then $a>0$ since $(0, \dots, 0) \in 0 \Log(r) \subset \Log(q_N)$.  Together with (1), we have $a \le_S f$, since 
$f-a =(q_N-a)+(f - q_N) \in \R_+[\x] + (r - 1) =S$.

Now let $h \in \R[\x]$ be given. Write $h=h_{+}+h_{-}$, where $h_{+}$ (resp. $h_{-}$) denotes the sum of the monomial terms of $h$ with positive (resp. negative) coefficients.  Then we have
$h_{+}\in\R_+[\x] $ and $h\le_S h_{+}$.
We are going to show that $ h_{+}\le_S Ca$ for all sufficiently large $C > 0$.  Obviously this holds if $h_{+}=0$. Thus we only consider the case when $h_+\neq  0$, so that $\Log(h_{+})\neq\emptyset$.   Let $M := \deg(h_{+})$ so that $\Log(h_{+}) \subset \Union_{d = 0}^{M} \N^n_d = \Union_{d = 0}^{M} d\N^n_1$. By Lemma \ref{lem: prep}, there exists some $g_{M} \in \R_+[\x]$ such that $g_M \congruent 1 \pmod{r - 1}$ and $\Log(g_M) = \Union_{d = 0}^M d\Log(r)$. 
Then for each real $C > 0$,
\begin{equation} \label{eq: UnitAsAlgInteriorPoint}
Ca - h_{+} \congruent Cag_M  - h_{+} \pmod{r - 1}.
\end{equation}
Write $h_{+}=\sum_{\alpha\in\Log(h_{+})}a_\alpha x^\alpha $ and 
$g_M=\sum_{\beta\in\Log(g_M)}b_\beta x^\beta$, so that the coefficients $a_\alpha $'s and $b_\beta$'s are all positive.
Since $ \Log(h_{+}) \subset \Union_{d = 0}^{M} d\N^n_1 \subset \Union_{d = 0}^M d\Log(r) = \Log(g_M)$, one easily sees that
\begin{equation}
Cag_M-h_{+}\in \R_+[\x]\quad \textrm{for all }C>\dfrac{ \max_{\alpha\in\Log(h_{+})}a_\alpha  }{a\min_{\beta\in\Log(g_M)}b_\beta  }.
\end{equation} 
Together with \eqref{eq: UnitAsAlgInteriorPoint}, it follows that $h_{+}\le_S Ca$ for all large $C$. In summary, we have $h\le_S h_{+}\le_S Ca\le_S Cf $ (so that $h\le_S Cf$) for all large $C$.  
Since $h \in \R[\x]$ is arbitrary, it follows that $f \in S^\circ$. 
\end{proof}

\begin{corollary} \label{cor: arch}
$S \subset \R[\x]$ is archimedean.
\end{corollary}

\begin{proof}
Note that $f=1$ satisfies Condition (iii) in Proposition \ref{lem: algebraicInteriorPoints} with $N=0$ and $q_N=1$.  
Thus Proposition \ref{lem: algebraicInteriorPoints} implies that $1 \in S^\circ$, i.e.\ $S \subset \R[\x]$ is archimedean.
\end{proof}

Now we are ready to give the proof of Theorem \ref{thm: mainTheorem} as follows:

\begin{proof}[Proof of Theorem \ref{thm: mainTheorem}]
Let $r \in \R_+[\x]$ with $\Log(r) \supseteq \MultiIndGrade{1}$ be as in Theorem \ref{thm: mainTheorem}. Let $S$ be as in \eqref{eq: Ssubsemialg}. Note that $\R$-algebra homomorphisms from $\R[\x] $ to $ \R$ correspond to evaluation at points in $\R^n$. For any point $x \in \R^n$, let $\hat{x} : \R[\x] \to \R$ be the associated $\R$-algebra homomorphism given by $\hat{x}(f) = f(x)$ for all $f \in \R[\x]$. Since $x_1, \dots, x_n, r - 1, 1 - r$ generate $S$ as an $\R_+$-semialgebra, it follows that $\varphi  \in \mathcal{X}(S)$ if and only if $\varphi=\hat{x}$ for some $x \in \{r = 1\}_+$. 
Therefore, for $f\in \R[\x]$, we have $f > 0$ on $\{r = 1\}_+$ if and only if $\varphi(f) > 0$ for all $\varphi \in \mathcal{X}(S)$ if and only if $f \in S^\circ$, by the Archimedean Representation Theorem (Theorem \ref{thm: ArchRepThm} with $A = \R[\x]$).
Therefore the equivalence of (a), (b), and (c) in Theorem \ref{thm: mainTheorem} follows from Proposition \ref{lem: algebraicInteriorPoints}.
%
%Now let $f \in \R[\x]$ be given. 
%
%\medskip\noindent
%(i) $\implies$ (ii):  If $f>0$
%on $\{r = 1\}_+$, then $\hat{x}(f) > 0$ for all $\hat{x} \in \mathcal{X}(S)$.  Therefore, by the Archimedean Representation Theorem, we have $f \in S^\circ$. Therefore, by Lemma \ref{lem: algebraicInteriorPoints}, there exists some $N_o\in\N$ such that for every integer $N\geq N_o$, there exists some $q_N \in \R_+[\x]$ with $\Log(q_N) = \Union_{d =0}^N d \Log(r)$ such that \eqref{eq: certificate} holds. 
%
%\medskip
%\noindent
%(ii) $\implies$ (iii): Obvious.
%
%\medskip
%\noindent
%(iii) $\implies$ (i):  Suppose that there exist some $N \in \N$ and some $q_N \in  \R_+[\x]$ with $\Union_{d = 0}^N d\Log(r)$ such that \eqref{eq: certificate} holds. Then, by Lemma \ref{lem: algebraicInteriorPoints}, $f \in S^\circ$. Therefore, by the Archimedean Representation Theorem (Theorem \ref{thm: ArchRepThm}) (with $A=\R[\x]$ and $S$ as in \eqref{eq: Ssubsemialg}), for all $\varphi $
%Take an arbitrary point 
%$x=(x_1, \dots, x_n)\in \{r = 1\}_+$. The congruence
%in \eqref{eq: certificate} implies that $f(x) = q_N(x)$. Write $q_N=\sum_{\beta\in \Log(q_N)}b_\beta x^\beta$, so that each $b_\beta>0$. Since $(0, \dots, 0) \in 0\Log(r) \subset \Union_{d = 0}^N d\Log(r) = \Log(q_N)$ and $x_1, \dots, x_n \ge 0$, it follows readily that $q_N(x) \ge b_{(0, \dots, 0)} > 0$. Therefore $f(x) = q_N(x) > 0$, which gives (i). 
\end{proof}

\paragraph{Acknowledgements.}  The first author would like to thank Professor Claus Scheiderer for an informal chat regarding this paper.

%===================================================================%

%\bigskip\noindent
%Colin Tan
\par\noindent
%{\it Email address}:  ???

\bigskip\noindent
{\it Corresponding author:}
\par\noindent
Wing-Keung To
\par\noindent
Department of Mathematics, National University of Singapore,
10 Lower Kent Ridge Road, Singapore 119076
\par\noindent
{\it Email address}:  mattowk@nus.edu.sg


\begin{thebibliography}{99}

\bibitem{BSTW}
Baldi, L., Sinn R., Telek M., Weiger, J.:
Toric extensions of P\'olya's theorem,
arxiv preprint: arXiv:math/9511201.


\bibitem{Barvinok2002}
Barvinok, A.: 
\textit{A course in convexity}, Graduate Studies in Mathematics, \textbf{54}, 
American Mathematical Society, Providence, RI, 2002. 

\bibitem{BerrWoermann2001}
Berr, R., W\"{o}rmann, T.:
Positive polynomials on compact sets, 
\textit{Manuscripta Math.} \textbf{104} (2001), no. 2, 135--143.

%\bibitem{bcr}
%J. Bochnak, M. Coste, M.-F.\ Roy:
%\emph{Real Algebraic Geometry}.
%Erg.\ Math.\ Grenzgeb.\ (3) \textbf{36}, Springer, Berlin, 1998.


%\bibitem{BSS12}
%Burgdorf, Sabine; Scheiderer, Claus; Schweighofer, Markus:
%Pure states, nonnegative polynomials and sums of squares. 
%\textit{Comment. Math. Helv.} \textbf{87} (2012), no. 1, 113--140.

%\bibitem{Catlin99}
%D. Catlin:
%The Bergman kernel and a theorem of Tian.
%Analysis and geometry in several complex variables (Katata, 1997),
%1--23, Trends Math., Birkh\"auser Boston, Boston, MA, 1999.

\bibitem{CatlinDangelo1996}
Catlin, D., D'Angelo, J.: 
A stabilization theorem for Hermitian forms and applications
to holomorphic mappings, \textit{Math.\ Res.\ Lett.} {\bf 3} (1996), no. 2, 149--166.

%\bibitem{CatlinDangelo1997}
%D. Catlin, J. D'Angelo:
%Positivity conditions for bihomogeneous polynomials.
%\textit{Math.\ Res.\ Lett.} {\bf 6} (1997),  555--567.

%\bibitem{CatlinDangelo1999}
%Catlin, D., D'Angelo, J.:
%An isometric imbedding theorem for holomorphic bundles.
%\textit{Math.\ Res.\ Lett.} {\bf 6} (1999),  43-60.

%\bibitem{CimpricZalar2013}
%\Cimpric, Jaka; Zalar, \Aljaz:
%Moment problems for operator polynomials. 
%\textit{J. Math. Anal. Appl.} \textbf{401} (2013), no. 1, 307--316.

\bibitem{DickinsonPovh2015}
Dickinson, P.J.C., Povh, J.: On an extension of P\'olya’s Positivstellensatz, \textit{J. Glob. Optim.} \textbf{61} (2015), 615--625. 

%\bibitem{DinhToLe2021}
%Dinh, Trung Hoa; Ho, Minh Toan; Le, Cong Trinh:
%Positivstellens\"{a}tze for polynomial matrices. 
%\textit{Positivity} \textbf{25} (2021), no. 4, 1295--1312.

%\bibitem{EffrosHandelmanShen1980}
%Effros, Edward G.; Handelman, David E.; Chao Liang Shen:
%Dimension groups and their affine representations. 
%\textit{Amer. J. Math.} \textbf{102} (1980), no. 2, 385--407.

%\bibitem{Eidelheit1936}
%Eidelheit, M.: Zur Theorie der konvenxen Mengen in linearen normierten %R\"aumen.
%\textit{Stud. Math.} \textbf{6} (1936), 104--111.

%\bibitem{GoodearlHandelman76}
%K. Goodearl, D. Handelman:
%Rank functions and $K_0$ of regular rings. 
%\textit{J.\ Pure Appl.\ Algebra} \textbf{7} (1976), 195--216. 

\bibitem{GoodearlHandelman1980}
Goodearl, K. R.; Handelman, D. E.: Metric completions of partially ordered abelian groups, \textit{Indiana Univ.\ Math.\ J.} \
\textbf{29} (1980), no. 6, 861--895.

%\bibitem{Handelman85}
%Handelman, David:
%Positive polynomials and product type actions of compact groups. 
%\textit{Mem. Amer. Math. Soc.} \textbf{54} (1985), no. 320.

%\bibitem{Handelman86}
%D. Handelman:
%Deciding eventual positivity of polynomials. Ergodic Theory Dynam. Systems %{\textbf{6}} (1986), 57--79.

\bibitem{Handelman1988}
Handelman,  D.:  
Representing polynomials by positive linear functions on compact convex polyhedra, 
\textit{Pac.\ J.\ Math.} \textbf{132} (1988), 35--62.

%\bibitem{Handelman92}
%D. Handelman:
%Polynomials with a positive power. Symbolic dynamics and its applications (New %Haven, CT, 1991), 229--230, Contemp. Math., 135, Amer. Math. Soc., Providence, %RI, 1992.

%\bibitem{Handelman12}
%D. Handelman:
%In praise of order units. J. Algebra Appl. {\textbf{11}}(6) (2012), 1250120, %16 pp . 

\bibitem{HHL1980}
Handelman, D., Higgs, D., Lawrence, J.: Directed abelian groups, countably continuous rings, and Rickart C*-algebras, \emph{J. London Math. Soc.} (2) \textbf{21} (1980), no. 2, 193--202.

\bibitem{HLP1952}
Hardy, G. H., Littlewood, J. E., P\'{o}lya, G.: 
\textit{Inequalities}, 2nd ed., 
Cambridge, at the University Press, 1952.

%\bibitem{Jacobson1964}
%Jacobson, N.: Lectures in abstract algebra III. Princeton: van Nostrand 1964.

%\bibitem{Kakutani1937}
%Kakutani, Shizuo:
%Ein Beweis des Satzes von M. Eidelheit \"{u}ber konvexe Mengen. 
%\textit{Proc. Imp. Acad. Tokyo} \textbf{13} (1937), no. 4, 93--94.

%\bibitem{KlepSchweighofer2010}
%Klep, Igor; Schweighofer, Markus:
%Pure states, positive matrix polynomials and sums of Hermitian squares. 
%\textit{Indiana Univ. Math. J.} \textbf{59} (2010), 857--874.

%\bibitem{Krivine1964a}
%Krivine, Jean-Louis:
%Quelques propri\'{e}t\'{e}s des pr\'{e}ordres dans les anneaux commutatifs unitaires.
%\textit{C. R. Acad. Sci. Paris} \textbf{258} (1964), 3417--3418.

%\bibitem{Krivine1964b}
%Krivine, J.-L.:
%Anneaux pr\'{e}ordonn\'{e}s.
%\textit{J. Analyse Math.} \textbf{12} (1964), 307--326.

%\bibitem{LeDu2018}
%L\^{e}, C\^{o}ng-Tr\`{i}nh; \Du, Th\.{i}-H\`{o}a-B\`{i}nh:
%Handelman's Positivstellensatz for polynomial matrices positive definite on polyhedra. 
%\textit{Positivity} \textbf{22} (2018), no. 2, 449--460.

\bibitem{Polya1928}
P\'olya, G.:
\"Uber positive Darstellung von Polynomen,
\textit{Vierteljschr.\ Naturforsch.\ Ges.\ Z\"urich} \textbf{73} (1928),
141--145, in \textit{Collected Papers} \textbf{2} (1974), MIT Press,
309--313.

\bibitem{PowersReznick2001}
Powers, V., Reznick, B.:
A new bound for P\'olya's theorem with applications to polynomials positive on polyhedra, 
\textit{J. Pure Appl. Algebra} \textbf{164} (2001), 221--229.

\bibitem{PrestelDelzell2001}
Prestel, A., Delzell, C. N.:
\textit{Positive polynomials: From Hilbert's 17th problem to real algebra}, 
Springer Monographs in Mathematics, Springer-Verlag, Berlin, 2001.

\bibitem{PutinarScheiderer2010}
Putinar, M., Scheiderer, C.:
Sums of Hermitian squares on pseudoconvex boundaries, 
\textit{Math. Res. Lett.} \textbf{17} (2010), no. 6, 1047--1053.

%\bibitem{PutinarScheiderer2014}
%Putinar, Mihai; Scheiderer, Claus:
%Quillen property of real algebraic varities.
%\textit{M\"unster Journal} \textbf{7} (2014), 671--696.

\bibitem{PutinarVasilescu1999}
Putinar, M., Vasilescu, F.-H.: Positive polynomials on semialgebraic sets, \textit{C. R. Acad. Sci. Paris S{\'e}r. I Math.} \textbf{328} (1999), no. 7, 585--589.

\bibitem{Quillen1968}
Quillen, D. G.:
On the representation of hermitian forms as sums of squares,
\textit{Invent. Math.} \textbf{5} (1968), 237--242.

\bibitem{Reznick1995}
Reznick, B.: Uniform denominators in Hilbert's seventeenth problem, \textit{Math Z.} \textbf{220} (1995), 75--97.

%\bibitem{Reznick1992}
%Reznick, Bruce: Sums of even powers of real linear forms. 
%\textit{Mem. Amer. Math. Soc.} \textbf{96} (1992), no. 463, viii+155 pp.

%\bibitem{Robinson75}
%A. Robinson:
%Algorithms in algebra. \textit{Model theory and algebra (A memorial tribute to Abraham Robinson)}, pp. 14--40. 
%Lecture Notes in Math., Vol. 498, Springer, Berlin, 1975.

%\bibitem{Sawin2023}
%Will Sawin:
%Modulo $x^2 + y^2 - 1$, is every homogeneous polynomial that is a square of a polynomial, necessarily of sum of squares of homogeneous polynomials?, 
%URL (version: 2023-03-30): https://mathoverflow.net/q/443838

%\bibitem{SchererHol2006}
%Scherer, C. W.; Hol, C. W. J.:
%Matrix sum-of-squares relaxations for robust semi-definite programs. 
%\textit{Math. Program.} \textbf{107} (2006), 189--211.



%\bibitem{sch:surf}
%C. Scheiderer:
%Sums of squares on real algebraic surfaces.
%Manuscr.\ math.\ \textbf{119}, 395--410 (2006).

\bibitem{Scheiderer2009}
Scheiderer, C.:
Positivity and sums of squares: a guide to recent results, 
 \textit{Emerging applications of algebraic geometry}, 271--324, 
 IMA Vol. Math. Appl., \textbf{149}, Springer, New York, 2009.

%\bibitem{Scheiderer12}
%C. Scheiderer:
%A Positivstellensatz for projective real varieties.
%Manuscr.\ math. {\bf 138} (2012), 73--88.

%\bibitem{Scheiderer2016}
%Scheiderer, Claus:
%An observation on positive definite forms.
%Preprint, arXiv:1602.03986.

\bibitem{Scheiderer2024}
Scheiderer, C.:
\textit{A course in real algebraic geometry - positivity and sums of squares}, Grad. Texts in Math., \textbf{303}, Springer, Cham, 2024.


%\bibitem{ScheidererTan2017}
%Scheiderer, Claus; Tan, Colin:
%A Positivstellensatz for forms on the positive orthant. 
%\textit{Arch. Math. (Basel)} \textbf{109} (2017), no. 2, 123--131.

%\bibitem{Schmuedgen1991}
%Schm\"{u}dgen, K.:
%The $K$-moment problem for compact semi-algebraic sets, 
%\textit{Math. Ann.} \textbf{289} (1991), no. 2, 203--206.

%\bibitem{sw}
%M. Schweighofer:
%Certificates for nonnegativity of polynomials with zeros on compact
%semialgebraic sets.
%Manuscr.\ math.\ \textbf{117}, 407--428 (2005).

%\bibitem{Tan2025}
%Tan, Colin:
%Archimedean Representation Theorem for modules over a commutative ring.
%\emph{Beitr.\ Algebra.\ Geom.} \textbf{66} (2025), 51--60.

%\bibitem{Woermann1998}
%W\"ormann, T.:
%Strikt positive Polynome in der semialgebraischen Geometrie, Doctoral %dissertation, \textit{Universit\"at Dortmund}, 1998.

\end{thebibliography}
\end{document}